\documentclass[11pt]{article}

\usepackage{times}
\usepackage{amsthm}
\usepackage{amsmath}
\usepackage{amssymb}
\usepackage{mathrsfs}

\def\o{\omega}
\def\O{\Omega}
\def\th{\theta}

\def\N{\mathbb{N}}

\def\B{{\mathbb B}}

\def\H{\mathcal H}

\def\R{\mathbb{R}}

\def\We{{\sf W}}
\def\W{{\mathfrak W}}
\def\e{{\sf e}}

\def\m{{\sf m}}

\def\wm{\widehat{\sf m}}

\def\im{\mathop{\mathsf{Im}}\nolimits}
\def\re{\mathop{\mathsf{Re}}\nolimits}

\def\({\left(}
\def\[{\left[}
\def\){\right)}
\def\]{\right]}

\def\d{\mathrm{d}}

\def\G{{\sf G}}
\def\wG{\widehat{\sf{G}}}
\def\p{\parallel}
\def\<{\langle}
\def\>{\rangle}

\newtheorem{Theorem}{Theorem}[section]
\newtheorem{Remark}[Theorem]{Remark}
\newtheorem{Lemma}[Theorem]{Lemma}

\newtheorem{Proposition}[Theorem]{Proposition}
\newtheorem{Definition}[Theorem]{Definition}

\newtheorem{Convention}[Theorem]{Convention}

 \numberwithin{equation}{section}
\begin{document}

\title{A Positive Quantization on Type I Locally Compact Groups}

\date{\today}

\author{M. M\u antoiu  \footnote{
\newline
\textbf{Adress:} Departamento de Matem\'aticas, Facultad de Ciencias, Universidad de Chile,
Las Palmeras 3425, Casilla 653, Santiago, Chile, 
\emph{e-mail:} mantoiu@uchile.cl
}}



\maketitle
	
\begin{abstract}
Let $\G$ be a unimodular type I second countable locally compact group and $\wG$ its unitary dual. Motivated by a recent pseudo-differential calculus, we develop a positive Berezin-type quantization with operator-valued symbols defined on $\wG\times\G$\,.
\end{abstract}





\section*{Introduction}\label{duci}

Let $\G$ be a locally compact group with unitary dual $\wG$\,, composed of unitary equivalence classes of irreducible representations. We assume it to be second countable, unimodular and type I, in order to have a manageable Fourier theory. In \cite{MR} a pseudodifferential calculus $a\mapsto{\sf Op}(a)$ has been proposed and studied, involving globally defined operator-valed symbols $a$ defined on the product between the group and its unitary dual. Having its roots in some undevelopped remarks of \cite{Ta}, this generalizes recent works on pseudo-differential operators on compact Lie groups \cite{DR1,Fi,RT,RT1,RW} or nilpotent Lie groups \cite{FR1,FR2} (see also other references cited therein). When the group has no Lie structure, it is not compact and (especially) it is non-commutative, the formalism needs delicate constructions for which we send to \cite{MR}. 

Since the paradigmatic initial exemple in constructing pseudo-differential calculi is the case $\G=\mathbb R^n$, cf \cite{Fo} for instance, it is natural to think of a positive quantization by Berezin-Toeplitz operators and this is the topic of the present paper. We recall that the Canonical Commutation Relations on the "phase-space" $\mathbb R^n\times\mathbb R^n$ can be codified by a family of Weyl operators ${\sf W}(\cdot,\cdot)$\,, forming a projective representation of $\mathbb R^n\times\mathbb R^n$\,. This family is useful both for pseudodifferential operators and for the positive quantization. 

If the group $\G$ is unimodular, type I, separable but non-commutative, we have shown in \cite{MR} that a Weyl family of operators still exists and is effective for constructing the pseudo-differential theory. In this article we use them for a Berezin-type quantization. If $\G$ is not commutative, the dual $\wG$ is not a group but the Mackey Borel structure and the Plancherel measure allow for a manageable integration theory. But another technical difficulty comes from the fact that the irreducible representations $\xi\in\wG$ are no longer $1$-dimensional (they can even have infinite dimension), which requires various  new tools (direct integrals, traces, non-commutative $L^p$-spaces). In particular, if $\xi$ and $\eta$ are not unitarily equivalent, the basic operators ${\sf W}(x,\xi)$ and ${\sf W}(y,\eta)$ live in different Hilbert spaces and cannot be composed. Another consequence is the need for operator-valued symbols. So we are very far from a theory relying on group-theoretical means.

The literature on coherent states and their associated positive quantizations in a group setting is huge and we cannot try to describe it here. Let us cite \cite{AAG,AFK,CR,Fo,Fu,GMP,Ha,Pe,Wo}, which in their turn cite many other important references. The terminology, points of view and applications are very diverse. In particular we call our main objects Berezin operators or Bargman transformations although the framework is very general and no holomorphy or Gaussian functions are in view. A different terminology (anti-Wick, localization) can be advocated.

The first section contains some preparations, including conventions and notations that will be used over the paper. The second one recalls from \cite{MR} the construction of the family of Weyl operators and the associated Fourier-Wigner transformation. In section 3 we introduce coherent states and in section 4 we define and study Berezin-type operators. They are connected to pseudo-differential operators in section 5. The last section presents the Bargman realization of the formalism, in which the Berezin operators have a Toeplitz form.

\section{Some conventions and notations}\label{fartais}

For a given (complex, separable) Hilbert space $\H$\,, one denotes by $\mathbb B(\H)$ the $C^*$-algebra of all linear bounded operators in $\H$ 
and by $\mathbb K(\H)$ the closed bi-sided $^*$-ideal of all the compact operators. 
The Hilbert-Schmidt operators form a two-sided $^*$-ideal $\mathbb B^2(\H)$\,, which is also a Hilbert space with the scalar product $\<A,B\>_{\mathbb B^2(\H)}:={\rm Tr}\!\(AB^*\)$\,. This Hilbert space is unitarily equivalent to the Hilbert tensor product $\H\otimes\overline{\H}$\,, where $\overline{\H}$ is the Hilbert space opposite to $\H$\,. 

Let $\G$ be a unimodular locally compact group with unit $\e$ and fixed Haar measure $\m$\,. For $p\in[1,\infty]$\,, the Lebesgue spaces $L^p(\G)\equiv L^p(\G;\m)$ are defined with respect to the Haar measure. The norm and the scalar product in $\H=L^2(\G)$ will be denoted, respectively, by $\p\!\cdot\!\p$ and $\<\cdot,\cdot\>$\,.

All over the article we are going to assume that the locally compact group $\G$ is {\it admissible}, i.e. it is unimodular, second countable and type I. For the concept of type I locally compact group, for examples and other harmonic analysis concepts and results that we use below we refer to \cite{Di,Fo1,Fu}; see also \cite[Sect. 2]{MR}. 

An element of ${\rm Irrep}(\G)$ is a strongly continuous irreducible unitary Hilbert space representation $\pi:\G\rightarrow\mathbb B(\H_\pi)$\,. Identifying unitarily equivalent representation, we set $\,\wG:={\rm Irrep(\G)}/_{\cong}$ and call it {\it the unitary dual of $\,\G$}\,. It is endowed with a Borel structure, called {\it the Mackey Borel structure}. The main consequence of admissibility is the existence of a measure $\wm$ on the unitary dual $\wG$ for which a Plancherel Theorem holds, in connection to the Fourier transformation that will be introduced below.

If $\G$ is Abelian all the irreducible representations are $1$-dimensional, the unitary dual $\wG$ is the Pontryagin dual group (composed of characters) and $\wm$ is a Haar measure. In the non-commutative case $\wG$ has no group structure.

It is known that there is a canonical $\wm$-measurable field $\big\{\,\H_\xi\mid\xi\in\wG\,\big\}$ of Hilbert spaces  and a measurable section $\wG\ni\xi\mapsto\pi_\xi\in{\rm Irrep(\G)}$ such that $\pi_\xi:\G\rightarrow\mathbb B(\H_\xi)$ is an irreducible representation belonging to the class $\xi$\,. Instead of $\pi_\xi$ we will write $\xi$\,. Using this identification, one introduces the direct integral Hilbert space 
\begin{equation*}\label{andy}
\mathscr B^2(\wG):=\int_{\wG}^\oplus\!\B^2(\H_\xi)\,d\wm(\xi)\,\cong\int_{\wG}^\oplus\!\H_\xi\otimes\overline\H_\xi\,d\wm(\xi)\,,
\end{equation*}
with the scalar product
\begin{equation*}\label{justin}
\<\phi_1,\phi_2\>_{\mathscr B^2(\wG)}:=\int_{\wG}\,\<\phi_1(\xi),\phi_2(\xi)\>_{\mathbb B^2(\H_\xi)}d\wm(\xi)=\int_{\wG}{\rm Tr}_\xi\!\[\phi_1(\xi)\phi_2(\xi)^*\]d\wm(\xi)\,,
\end{equation*}
where ${\rm Tr}_\xi$ is the usual trace in $\mathbb B(\H_\xi)$\,. We also recall that the von Neumann algebra of decomposable operators $\mathscr B(\wG):=\int_{\wG}^\oplus\B(\H_\xi)\,d\wm(\xi)$ acts to the left and to the right in the Hilbert space $\mathscr B^2(\wG)$ in an obvious way.

{\it The Fourier transform} of $u\in L^1(\G)$ is given in weak sense by
\begin{equation*}\label{ion}
(\mathscr F u)(\xi):=\int_\G u(x)\pi_\xi(x)^*d\m(x)\equiv\int_\G u(x)\xi(x)^*d\m(x) \in\mathbb B(\H_\xi)\,.
\end{equation*} 
As a map $\,\mathscr F:L^1(\G)\rightarrow\mathscr B(\wG)$ it is linear, contractive and injective. It also defines a unitary isomorphism $\mathscr F:L^2(\G)\rightarrow \mathscr B^2(\wG)$\,; this is the mentioned generalization of Plancherel's Theorem to (maybe non-commutative) admissible groups. 

It is useful to note that for every $\xi\in\wG$ the space $L^2\big[\G;\mathbb B(\H_\xi)\big]$ has the structure of a Hilbert-module over the unital $C^*$-algebra
$\mathbb B(\H_\xi)$ with operation
\begin{equation*}\label{doih}
\mathbb B(\H_\xi)\times L^2\big[\G;\mathbb B(\H_\xi)\big]\ni(A,a)\mapsto (A\cdot a)(\cdot):=A a(\cdot)\in L^2\big[\G;\mathbb B(\H_\xi)\big]\,,
\end{equation*}
and $\mathbb B(H_\xi)$-valued inner product
\begin{equation*}\label{treih}
\<\cdot\!\mid\!\cdot\>_\xi:L^2\big[\G;\mathbb B(\H_\xi)\big]\times L^2\big[\G;\mathbb B(\H_\xi)\big]\to\mathbb B(\H_\xi)\,,\quad\<a\!\mid\!b\>_\xi:=\int_\G a(x)b(x)^*d\m(x)\,.
\end{equation*}
Note the relations
\begin{equation*}\label{haha}
\<A\cdot a\!\mid\! b\>_\xi=A\<a\!\mid\! b\>_\xi\,,\quad\<a\!\mid\! A\cdot b\>_\xi=\<a\!\mid\!b\>_\xi A^*,\quad\<a\!\mid\!b\>^*_\xi=\<b\!\mid\!a\>_\xi
\end{equation*}
valid for $a,b\in L^2\big[\G;\mathbb B(\H_\xi)\big]$ and $A\in\mathbb B(\H_\xi)$\,. By applying the definitions and the usual Cauchy-Schwartz inequality one obtains immediately
\begin{equation*}\label{cs}
\big\Vert \<a\!\mid\!b\>_\xi\big\Vert_{\mathbb B(\H_\xi)}\le\,\p\!a\!\p_{L^2[\G;\mathbb B(\H_\xi)]}\,\p\!b\!\p_{L^2[\G;\mathbb B(\H_\xi)]}\,.
\end{equation*}

Most often one works with $a=u\otimes A$ and $b=v\otimes B$ for $u,v\in L^2(\G)$ and $A,B\in\mathbb B(\H_\xi)$\,, where for example $(u\otimes A)(x):=u(x)A$ for every $x\in\G$\,. Then
\begin{equation}\label{then}
\<u\otimes A\!\mid\!v\otimes B\>_\xi=\<u,v\>AB^*.
\end{equation}
Later we are going to need a version of the Parseval identity for the $\mathbb B(\H_\xi)$-valued inner product. Let $\{w_k\}_{k}$ be an ortonormal base of the Hilbert space $L^2(\G)$ (it is separable, since $\G$ has been assumed second countable). Then, by \eqref{then} and the usual form of the Parseval identity in $L^2(\G)$ one has
$$
\begin{aligned}
\sum_k\,\<u\otimes A\!\mid\!w_k\otimes 1_\xi\>_\xi\,\<w_k\otimes 1_\xi\!\mid\!v\otimes B\>_\xi&=\sum_k\<u,w_k\>\<w_k,v\>AB^*\\
&=\<u,v\>AB^*=\<u\otimes A\!\mid\!v\otimes B\>_\xi\,.
\end{aligned}
$$
It follows easily from this identity that
\begin{equation}\label{hermes}
\sum_k\,\<a\!\mid\!w_k\otimes 1_\xi\>_\xi\,\<w_k\otimes 1_\xi\!\mid\!b\>_\xi=\<a\!\mid\!b\>_\xi\,,\quad\forall\,a,b\in L^2\big[\G;\mathbb B(\H_\xi)\big]\,.
\end{equation}

\begin{Convention}\label{convention}
{\rm Let us fix some $\xi\in\wG$\,. If $\,T\in\mathbb B\big[L^2(\G;\H_\xi)\big]$ one sets 
\begin{equation*}\label{ona}
T:L^2(\G)\to\mathbb B\big[\H_\xi,L^2(\G;\H_\xi)\big]\,,\quad (Tu)\varphi_\xi:=T(u\otimes\varphi_\xi)\,.
\end{equation*}
By a change of order of variables, this can be reinterpreted as
\begin{equation*}\label{dona}
Tu:\G\to\mathbb B(\H_\xi)\,,\quad[(Tu)(q)]\varphi_\xi:=\big[(Tu)\varphi_\xi\big](q)=\big[T(u\otimes\varphi_\xi)\big](q)\,.
\end{equation*}
It is not always true in this realization that $Tu\in L^2\big[\G;\mathbb B(\H_\xi)\big]$ for every $u\in L^2(\G)$\,, but this does happen below in some interesting cases. If, by this interpretation, one gets a bounded operator $T:L^2(\G)\to L^2\big[\G;\mathbb B(\H_\xi)\big]$\,, we say that $T$ is {\it a manageable operator}.

In such a situation one gets $\<Tu\!\mid\!v\otimes 1_\xi\>_\xi\in \mathbb B(\H_\xi)$ given by
\begin{equation*}\label{cincona}
\<Tu\!\mid\!v\otimes 1_\xi\>_\xi\,\varphi_\xi=\int_\G \,[(Tu)(y)]\varphi_\xi\,\overline{v(y)}\,d\m(y)=\int_{\G}\,[T(u\otimes\varphi_\xi)](y)\,\overline{v(y)}\,d\m(y)\in\H_\xi
\end{equation*}
and leading  for $u,v\in L^2(\G)$ and $\varphi_\xi,\psi_\xi\in\H_\xi$ to
\begin{equation}\label{tensiune}
\Big<\<Tu\!\mid\!v\otimes 1_\xi\>_\xi\,\varphi_\xi,\psi_\xi\Big>_{\!\H_\xi}\!=\big<T(u\otimes\varphi_\xi),v\otimes\psi_\xi\big>_{L^2(\G;\H_\xi)}\,.
\end{equation}
Then it is easy to check the useful formula
\begin{equation*}\label{formula}
\<Tu\!\mid\!v\otimes 1_\xi\>_\xi=\<u\otimes 1_\xi\!\mid\!T^*v\>_\xi\,,\quad u,v\in L^2(\G)\,.
\end{equation*}
}
\end{Convention}

\section{Weyl systems and the Fourier-Wigner transformation}\label{fourtein}

This Section is to a great extent a review of some constructions from \cite{MR}. Our definitions of the Weyl system and the Fourier-Wigner transformation are slightly different from those in \cite{MR}, but the properties are the same.

\begin{Definition}\label{siegmund}
For any $(\xi,x)\in\wG\times\G$ one defines {\rm the Weyl operator} ${\sf W}(\xi,x)$ in the Hilbert space $L^2(\G;\H_\xi)\equiv L^2(\G)\otimes\H_\xi$ by
\begin{equation*}\label{frieda}
\[\We(\xi,x)\Psi_\xi\]\!(q):=\xi(q)\big[\Psi_\xi(qx^{-1})\big]\,.
\end{equation*}
\end{Definition}

The terminology is inspired by the particular case $\G=\R^n$ \cite{Fo}. The operator $\We(\xi,x)$ is unitary, with adjoint 
\begin{equation*}\label{adjunctul}
\[\We(\xi,x)^*\Psi_\xi\]\!(q)=\xi(qx)^*\!\[\Psi_\xi(qx)\].
\end{equation*} 
In general $\We(\xi,x)$ cannot be composed with $\We(\eta,y)$ for $\xi\ne\eta$\,, because they act in different spaces. 

\begin{Lemma}\label{abil}
For every $(\xi,x)\in\mathscr X$ the operators $\,\We(\xi,x)$ and $\,\We(\xi,x)^*$ are manageable. 
\end{Lemma}

\begin{proof}
Using the conventions of the preceding section, we deal with $\,\We(\xi,x)^*$; a similar argument works for $\,\We(\xi,x)$\,. For $u\in L^2(\G)$ one has
$$
\begin{aligned}
\p\!\We(\xi,x)^*u\!\p^2_{L^2[\G;\mathbb B(\H_\xi)]}&=\int_\G\p\!\big[\We(\xi,x)^*u\big](y)\!\p^2_{\mathbb B(\H_\xi)}\!d\m(y)\\
&=\int_\G\,\sup_{\p\varphi_\xi\p=1}\!\p\!\big[\We(\xi,x)^*(u\otimes\varphi_\xi)\big](y)\!\p^2_{\H_\xi}\!d\m(y)\\
&=\int_\G\,\sup_{\p\varphi_\xi\p=1}\!\p\!u(yx)\,\xi(yx)^*\varphi_\xi\!\p^2_{\H_\xi}\!d\m(y)\\
&=\int_\G\,|u(yx)|^2\sup_{\p\varphi_\xi\p=1}\!\p\!\xi(yx)^*\varphi_\xi\!\p^2_{\H_\xi}\!d\m(y)=\,\p\!u\!\p^2,
\end{aligned}
$$
so $\We(\xi,x)^*$ can be regarded as an isometry $:L^2(\G)\mapsto L^2[\G;\mathbb B(\H_\xi)]$\,.
\end{proof}

\begin{Definition}\label{sigemund}
For $(\xi,x)\in\wG\times\G$ and $u,v\in L^2(\G)$ one sets 
\begin{equation*}\label{leo}
\W_{u,v}(\xi,x):=\big<\We(\xi,x)^*u\!\mid\!v\otimes 1_\xi\big>_\xi=\big<u\otimes 1_\xi\!\mid\!\We(\xi,x)v\big>_\xi\in\mathbb B\big(\H_\xi\big)\,.
\end{equation*}
The map $(u,v)\mapsto\W_{u,v}\equiv\W(u\otimes v)$ is called {\rm the Fourier-Wigner transformation}.
\end{Definition}

By \eqref{tensiune}, an equivalent definition of $\W_{u,v}(\xi,x)$ is via
\begin{equation*}\label{suspinare}
\left<\W_{u,v}(\xi,x)\varphi_\xi,\psi_\xi\right>_{\H_\xi}=\big<\We(\xi,x)^*(u\otimes\varphi_\xi),v\otimes\psi_\xi\big>_{L^2(\G;\H_\xi)}\,,
\end{equation*}
valid for $u,v\in L^2(\G)\,,\,\varphi_\xi,\psi_\xi\in\H_\xi\,,\,(\xi,x)\in\mathscr X$\,. Explicitly one has
\begin{equation}\label{wilhelm}
\W_{u,v}(\xi,x)=\int_\G u(z)\,\overline{v(zx^{-1})\,}\xi(z)^*d\m(z)=\Big((\mathscr F\otimes{\sf id})\big[(u\otimes\overline v)\circ\vartheta\big]\Big)(\xi,x)\,,
\end{equation}
which, denoting by  $\gamma$ the composition with the change of variables $\gamma(z,x)=\big(z,zx^{-1}\big)$\,, can be written
\begin{equation*}\label{ostrogot}
\W_{u,v}=(\mathscr F\otimes{\sf id})\gamma(u\otimes \overline v)\,.
\end{equation*}
Since $\,\gamma:L^2(\G\times\G)\to L^2(\G\times\G)$ and $\,\mathscr F\otimes{\sf id}:L^2(\G\times\G)\to\mathscr B^2(\mathscr X)$ are unitary, one has

\begin{Proposition}\label{mataus} 
\,The mapping $(u,v)\to\W_{u,v}$ defines a unitary transformation (denoted by the same symbol) 
$$
\W:L^2(\G)\otimes \overline{L^2(\G)}\rightarrow\mathscr B^{2,2}(\mathscr X):= \mathscr B^2(\wG)\otimes L^2(\G)\,. 
$$ 
\end{Proposition}

For $p\in[1,\infty)$\,, one defines the Banach space $\mathscr B^p(\wG)$ as the family of measurable fields $\,f\equiv\big(f(\xi)\big)_{\xi\in\wG}\,$ for which $f(\xi)$ belongs to the Schatten-von Neumann class $\,\mathbb B^p(\H_\xi)$ for almost every $\xi$ and
\begin{equation*}\label{bp}
\p\!f\!\p_{\mathscr B^p(\wG)}\,:=\Big(\int_{\wG}\p\!f(\xi)\!\p_{\mathbb B^p(\H_\xi)}^p\!d\wm(\xi)\Big)^{1/p}<\infty\,.
\end{equation*}
It is shown in \cite{Fu2} that if $p\in[1,2]$ and $1/p+1/p'=1$ the Fourier transform $\mathscr F$ is a well-defined linear contraction from $L^p(\G)$ to $\mathscr B^{p'}\!(\wG)$\,. 

Then we introduce the Banach space $\mathscr B^{p,p}(\mathscr X):=L^p\big[\G;\mathscr B^p(\wG\,)\big]$ with the norm
\begin{equation*}\label{scarbos}
\begin{aligned}
\p\!F\!\p_{\mathscr B^{p,p}(\mathscr X)}\,:=\Big(\int_\G \p\!F(x)\!\p^p_{\mathscr B^{p}(\wG)}\!d\m(x)\Big)^{1/p}\!=\Big(\int_\G \Big[\int_{\wG}\p\!F(\xi,x)\!\p^p_{\mathbb B^{p}(\H_\xi)}\!d\wm(\xi)\Big]d\m(x)\Big)^{1/p}\!,
\end{aligned}
\end{equation*}
where the notation $\,F(\xi,x):=[F(x)](\xi)$ has been used. It can be shown that $\mathscr B^{1,1}(\mathscr X)\cong \mathscr B^1(\wG\,)\,\overline\otimes\,L^1(\G)$ (projective completed tensor product), $\mathscr B^{2,2}(\mathscr X)=\mathscr B^2(\wG\,)\otimes L^2(\G)$ (Hilbert tensor product) and $\mathscr B^{\infty,\infty}(\mathscr X)\cong\mathscr B^\infty(\wG\,)\,\widetilde\otimes\,L^\infty(\G)$ (von Neumann tensor product). The usual modification is required if $p=\infty$\,:
\begin{equation*}\label{required}
\p\!F\!\p_{\mathscr B^{\infty,\infty}(\mathscr X)}\,:=\underset{x\in\G}{\rm ess\,sup}\!\p\!F(x)\!\p_{\mathscr B^{\infty}(\wG)}\,=\underset{x\in\G,\,\xi\in\wG}{\rm ess\,sup}\p\!F(\xi,x)\!\p_{\mathbb B(\H_\xi)}.
\end{equation*}

\begin{Remark}\label{trecut}
{\rm The Banach spaces $\mathscr B^{p,p}(\mathscr X)$ have been put in \cite{MR} in the perspective of non-commutative $L^p$-spaces \cite{PX,Xu}\,; they are associated to a natural direct integral trace on the von Neumann algebra $\mathscr B^{\infty,\infty}(\mathscr X)$\,. Therefore, they have good interpolation and duality properties that will be used below. In particular, if $1/p+1/q=1$\,, $F\in\mathscr B^{p,p}(\mathscr X)$ and $H\in\mathscr B^{q,q}(\mathscr X)$\,, one defines 
\begin{equation*}\label{trabel}
\<F,H\>_{(\mathscr X)}:=\int_\mathscr X\!{\rm Tr}_\xi\big[F(X)H(X)^*\big]d\mu(X)
\end{equation*}
and get a sesquilinear form satisfying
\begin{equation*}\label{dracu}
\big|\<F,H\>_{(\mathscr X)}\big|\le\,\p\!F\!\p_{\mathscr B^{p,p}(\mathscr X)}\p\!H\!\p_{\mathscr B^{q,q}(\mathscr X)}\,.
\end{equation*}
}
\end{Remark}

\begin{Proposition}\label{matraus} 
\,For every $p\in[2,\infty]$ one has a linear contraction $\W:L^2(\G)\otimes \overline{L^2(\G)}\rightarrow\mathscr B^{p,p}(\mathscr X)$\,.
\end{Proposition}

\begin{proof}
See \cite[Sect. 3.4]{MR} for the proof, with slightly different notations and conventions.
\end{proof}

\section{A family of projections}\label{firorton}

It will be convenient to use notations as $\,X=(\xi,x)\,,Y=(\eta,y)\in\mathscr X:=\wG\times\G$ and the product measure $\,d\mu(X)=d\m(x) d\wm(\xi)$ on $\mathscr X$. We go on applying Convention \ref{convention}.

\begin{Definition}\label{coyrent}
Let us fix an element $\o\in L^2(\G)$\,. For $X\in\mathscr X$ we define $\O(X):\H_\xi\to L^2(\G;\H_\xi)$ by
\begin{equation}\label{stass}
\big[\O(X)\varphi_\xi\big](z)\equiv\big[\O(X)(z)\big]\varphi_\xi:=\big[\We(X)(\o\otimes\varphi_\xi)\big](z)=\o(zx^{-1})\xi(z)\varphi_\xi\,.
\end{equation}
\end{Definition}

The family $\,\{\O(X)\mid X\in\mathscr X\}$ can be thought as a subset of $L^2\big[\G;\mathbb B(\H_\xi)\big]$\,, with 
\begin{equation}\label{stim}
\p\!\O(X)\!\p_{L^2[\G;\mathbb B(\H_\xi)]}\,=\,\p\!\o\!\p\,,\ \ \ \forall\,X\in\mathscr X.
\end{equation} 
Actually, admitting the interpretation $\,\We(\xi,x)\in L^2\big[\G;\mathbb B(\H_\xi)\big]$ for this manageable operator (cf. Convention \ref{convention} and Lemma \ref{abil}), we can write $\O(X)=\We(X)\o$\,.
By using \eqref{wilhelm}, \eqref{stass} and the inner product one sees that
\begin{equation}\label{legatura}
\W_{u,\o}(X)=\<u\otimes 1_\xi\!\mid\!\O(X)\>_\xi=\<\We(X)^*u\!\mid\!\o\otimes 1_\xi\>_\xi\,.
\end{equation}
The family $\,\{\O(X)\mid X\in\mathscr X\}$ could be called {\it the family of coherent states generated by the vector $\o$}\,. Note that in the non-commutative case these "coherent states" do not belong to the same space.

\begin{Remark}\label{florina}
{\rm When $\G$ is Abelian $\xi$ is a character, one has $\H_\xi=\mathbb C$ and $\mathscr B^{2,2}(\mathscr X)=L^2(\wG\times\G)$\,. Thus $\O(X)\equiv\o(X)\in\mathbb B\!\[\mathbb C;L^2(\G;\mathbb C)\]\cong L^2(\G)$ and we can write simply $\o(X)=\We(X)\o$\,; no special convention is needed and one has $\o\big(0_\G,0_{\wG}\big)=\o$\,. In such a familiar framework $\O(X)$ is a "traditional" coherent state and one has $\,\W_{u,\o}(\cdot)=\<u,\o(\cdot)\>$\,.
}
\end{Remark}

\begin{Definition}\label{schimb}
For every $X=(\xi,x)\in\mathscr X$\,, the operator $\,{\sf Pr}_\o(X):L^2(\G)\otimes\H_\xi\to L^2(\G)\otimes\H_\xi\,$ is given on elementary vectors $u\otimes\varphi_\xi\in L^2(\G)\otimes\H_\xi$ by
\begin{equation*}\label{note}
{\sf Pr}_\o(X)(u\otimes\varphi_\xi):=\We(X)\big[{\sf id}_{L^2(\G)}\otimes\<\We(X)^*u\!\mid\!\o\otimes 1_\xi\>_\xi\big](\o\otimes\varphi_\xi)\,.
\end{equation*}
\end{Definition}

Making use of the Fourier-Wigner transform, this can also be written
\begin{equation*}\label{notte}
{\sf Pr}_\o(X)(u\otimes\varphi_\xi)=\We(X)\big[\o\otimes\W_{u,\o}(X)\varphi_\xi\big]\,.
\end{equation*}

The full definition is still to be completed; see \eqref{generaliz}. This partial form is justified by the commutative case:
If $\G$ is Abelian, each $\H_\xi$ is one-dimensional, $\o(X)=\We(X)\o$ belongs to $L^2(\G)$ and ${\sf Pr}_\o(X)$ is the rank one projection associated to the vector $\o(X)$\,, i.e. 
\begin{equation*}\label{comabel}
{\sf Pr}_\o(X)u=\<u,\We(X)\o\>\,\We(X)\o\,.
\end{equation*}

A computation using the explicit form of $\We(X)$\,, $\W_{u,\o}$ and $\o(X)$ gives
$$
\begin{aligned}
\Big[{\sf Pr}_\o(X)(u\otimes\varphi_\xi)\Big](q)&=\Big(\We(X)\big[\o\otimes\,\W_{u,\o}(X)\varphi_\xi\big]\Big)(q)\\
&=\,\o(qx^{-1})\,\xi(q)\,\W_{u,\o}(X)\varphi_\xi\\
&=\,\o(qx^{-1})\,\xi(q)\!\int_\G u(y)\,\overline{\o(yx^{-1})}\,\xi(y)^*\varphi_\xi\,d\m(y)\\
&=\,[\O(X)](q)\!\int_\G u(y)\,[\O(X)](y)^*\varphi_\xi\,d\m(y)\\
&=\,[\O(X)](q)\!\int_\G\,[\O(X)](y)^*\big[(u\otimes\varphi_\xi)(y)\big]\,d\m(y)\,.
\end{aligned}
$$
This leads to the final full expresion of our operators: for each $\Phi_\xi\in L^2(\G;\H_\xi)$ one sets
\begin{equation}\label{generaliz}
\begin{aligned}
\Big[{\sf Pr}_\o(X)\Phi_\xi\Big](q)&=[\O(X)](q)\int_\G\,[\O(X)](y)^*\big[\Phi_\xi(y)\big]d\m(y)\\
&=\o(qx^{-1})\int_\G\,\overline{\o(yx^{-1})}\,\xi(qy^{-1})\big[\Phi_\xi(y)\big]\,d\m(y)\,.
\end{aligned}
\end{equation}

\begin{Proposition}\label{pozitiv}
For every $X\in\mathscr X$ the operator $\,{\sf Pr}_\o(X)$ is linear bounded and positive. One has
\begin{equation}\label{tristan}
{\sf Pr}_\o(X)^2=\,\p\!\o\!\p^2\!{\sf Pr}_\o(X)\,,
\end{equation}
\begin{equation*}\label{thebound}
\p\!{\sf Pr}_\o(X)\!\p_{\mathbb B[L^2(\G;\H_\xi)]}\,=\,\p\!\o\!\p^2.
\end{equation*}
In particular, if $\o$ is normalized, ${\sf Pr}_\o(X)$ is a self-adjoint projection.
\end{Proposition}

\begin{proof}
Let $\Phi_\xi,\Psi_\xi\in L^2(\G;\H_\xi)$\,. One has
$$
\begin{aligned}
\big\<{\sf Pr}_\o(X)\Phi_\xi,&\Psi_\xi\big\>_{L^2(\G;\H_\xi)}=\int_\G\big\<\big[{\sf Pr}_\o(X)\Phi_\xi\big](q),\Psi_\xi(q)\big\>_{\H_\xi}d\m(q)\\
&=\int_\G\Big\<[\O(X)](q)\int_\G [\O(X)](y)^*\big[\Phi_\xi(y)\big]\,d\m(y),\Psi_\xi(q)\Big\>_{\!\H_\xi}d\m(q)\\
&= \Big\<\int_\G\,[\O(X)](y)^*\big[\Phi_\xi(y)\big]d\m(y),\int_\G\,[\O(X)](q)^*\big[\Psi_\xi(q)\big]\d\m(q)\Big\>_{\!\H_\xi}
\end{aligned}
$$
and self-adjointness follows immediately. Positivity also folows taking $\Phi_\xi=\Psi_\xi$\,.

A new computation, using the second form of ${\sf Pr}_\o(X)$ in \eqref{generaliz} this time, leads to
$$
\begin{aligned}
\Big({\sf Pr}_\o(X)&\big[{\sf Pr}_\o(X)\Phi_\xi\big]\Big)(q)=\,\o(qx^{-1})\int_\G \,\overline{\o(zx^{-1})}\,\xi(qz^{-1})\big[\big({\sf Pr}_\o(X)\Phi_\xi\big)(z)\big]\,d\m(z)\\
&=\,\o(qx^{-1})\int_\G\, \overline{\o(zx^{-1})}\,\xi(qz^{-1})\Big[\o(zx^{-1})\int_\G\,\overline{\o(yx^{-1})}\,\xi(zy^{-1})\big[\Phi_\xi(y)\big]\,d\m(y)\Big]d\m(z)\\
&=\,\o(qx^{-1})\int_\G \Big[\int_\G\,\overline{\o(zx^{-1})}\o(zx^{-1})\,d\m(z)\Big]\,\overline{\o(yx^{-1})}\,\xi(qy^{-1})\big[\Phi_\xi(y)\big]d\m(y)\\
&=\;\p\!\o\!\p^2\!\big[{\sf Pr}_\o(X)\Phi_\xi\big](q)\,.
\end{aligned}
$$
Thus we have \eqref{tristan} and clearly this finishes the proof.
\end{proof}

It is also useful to compute for $u,v\in L^2(\G)$\,, $X=(\xi,x)\in\mathscr X$ and $\varphi_\xi\in\H_\xi$
$$
\begin{aligned}
\<{\sf Pr}_\o(X)u\!\mid\!v\otimes 1_\xi\>_\xi\,\varphi_\xi&=\int_\G\,\big[{\sf Pr}_\o(X)(u\otimes\varphi_\xi)\big](z)\,\overline{v(z)}\,d\m(z)\\
&=\int_\G\,\overline{v(z)}\,[\O(X)](z)\,d\m(z)\!\int_\G u(y)[\O(X)](y)^*\varphi_\xi\,d\m(y)\,,
\end{aligned}
$$
which reads
\begin{equation}\label{reliable}
\<{\sf Pr}_\o(X)u\!\mid\!v\otimes 1_\xi\>_\xi=\<\O(X)\!\mid\!v\otimes 1_\xi\>_\xi\,\<u\otimes 1_\xi\!\mid\!\O(X)\>_\xi\,.
\end{equation}
Taking \eqref{legatura} into account, this can also be written
\begin{equation}\label{vartej}
\<{\sf Pr}_\o(X)u\!\mid\!v\otimes 1_\xi\>_\xi=\W_{v,\o}(X)^*\,\W_{u,\o}(X)\equiv\big(\W_{v,\o}^{\,*}\,\W_{u,\o}\big)(X)\,.
\end{equation}

\section{The Berezin quantization}\label{firton}

\begin{Definition}\label{uwe}
Let $\,\o\in L^2(\G)$ be fixed. We define formally, for any $F\in\mathscr B^{\infty,\infty}(\mathscr X)$\,, the operator in $L^2(\G)$
\begin{equation*}\label{albina}
{\sf Ber}_\o(F):=\int_{\mathscr X}\!{\rm Tr}_\xi\big[F(X){\sf Pr}_\o(X)\big]d\mu(X)
\end{equation*}
and call it {\rm the Berezin operator associated to the symbol $F$ and the vector $\o$}\,.
\end{Definition}

This should be taken in weak sense, i.e. for any $u,v\in L^2(\G)$ one sets
\begin{equation}\label{viespe}
\begin{aligned}
\big\<{\sf Ber}_\o(F)u,v\big\>:=&\int_{\mathscr X}\!{\rm Tr}_\xi\!\big[F(X)\,\big\<{\sf Pr}_\o(X)u\!\mid\!v\otimes 1_\xi\big\>_\xi\big]d\mu(X)\\
=&\int_{\mathscr X}\!{\rm Tr}_\xi\!\big[F(X)\,\W_{v,\o}(X)^*\,\W_{u,\o}(X)\big]d\mu(X)\,.
\end{aligned}
\end{equation}
This last expression, for which \eqref{vartej} has been used, can also be written
\begin{equation}\label{monstru}
\big\<{\sf Ber}_\o(F)u,v\big\>=\big\<F,\W_{u,\o}^{\,*}\,\W_{v,\o}\big\>_{(\mathscr X)}.
\end{equation}
Obviously, the correspondence $F\mapsto{\sf Ber}_\o(F)$ is linear. Using \eqref{viespe}, it is also straightforward to check that ${\sf Ber}_\o(F)^*={\sf Ber}_\o(F^\star)$\,, where $F^\star(X):=F(X)^*$ (adjoint in $\mathbb B(\H_\xi)$)\,.

\begin{Remark}\label{unnu}
{\rm Let us define $\,\mathbf 1\in\mathscr B^{\infty,\infty}(\mathscr X)$ by $\mathbf 1(\xi,x):=1_\xi$ (the identity operator in $\H_\xi$)\,. Then, by Proposition \ref{mataus},
$$
\big\<{\sf Ber}_\o(\mathbf 1)u,v\big\>=\int_{\mathscr X}\!{\rm Tr}_\xi\!\big[\W_{v,\o}(X)^*\,\W_{u,\o}(X)\big]d\mu(X)=\big\<\W_{u,\o}(X),\W_{v,\o}(X)\big\>_{(\mathscr X)}\!=\,\p\!\o\!\p^2\!\<u,v\>\,.
$$
Thus, if $\o$ is a normalized vector, we have ${\sf Ber}_\o(\mathbf 1)=1_{L^2(\G)}$ and one can write formally "the overcompleteness relation"
\begin{equation*}\label{overcompleteness}
\int_{\mathscr X}\!{\rm Tr}_\xi\big[{\sf Pr}_\o(X)\big]d\mu(X)=1_{L^2(\G)}\,.
\end{equation*}
}
\end{Remark}

We gather the most important properties of the Berezin operators in the next result. We say that $F\in\mathscr B^{p,p}(\mathscr X)$ is {\it positive} if for $\mu$-almost every $X\in\mathscr X$ the operator $F(X)\in\mathbb B^p(\H_\xi)\subset\mathbb B(\H_\xi)$ is positive. If $F\in\mathscr B^{1,1}(\mathscr X)$ we write $\mathbf{Tr}(F):=\int_\mathscr X\!{\rm Tr}_\xi\!\big[F(X)\big]d\mu(X)$\,. This is essentially the trace used in defining the non-commutative $L^p$-spaces $\big\{\mathscr B^{p,p}(\mathscr X)\mid p\in[1,\infty]\big\}$\,, cf. \cite{PX,Xu,MR}; it has a nice realization in terms of the Berezin quantization.

\begin{Theorem}\label{bondar}
For every $s\in[1,\infty]$ one has a linear bounded map ${\sf Ber}_\o:\mathscr B^{s,s}(\mathscr X)\to\mathbb B^s\big[L^2(\G)\big]$ satisfying
\begin{equation}\label{aceea}
\p\!{\sf Ber}_\o(F)\!\p_{\mathbb B^s[L^2(\G)]}\,\le\;4^{1/s}\p\!F\!\p_{\mathscr B^{s,s}}\,\p\!\o\!\p^2.
\end{equation} 
In particular, if $\,F\in\mathscr B^{1,1}(\mathscr X)$\,, then ${\sf Ber}_\o(F)$ is a trace-class operator with
\begin{equation*}\label{musca}
{\rm Tr}\big[{\sf Ber}_\o(F)\big]=\;\p\!\o\!\p^2\mathbf{Tr}(F)\,.
\end{equation*}
If $\,F\in\mathscr B^{s,s}(\mathscr X)$ is positive, then ${\sf Ber}_\o(F)$ is a positive operator in $L^2(\G)$\,.
\end{Theorem}

\begin{proof}
1. The non-commutative $L^p$-spaces $\mathscr B^{p,p}(\mathscr X)$ satisfy H\"older-type inequalities \cite{Fu2}, meaning that if $p,q,r\in[1,\infty]$ with $1/p+1/q=1/r$\,, then, with respect to pointwise multiplication, one has 
\begin{equation*}\label{poetic}
\mathscr B^{p,p}(\mathscr X)\,\mathscr B^{q,q}(\mathscr X)\subset\mathscr B^{r,r}(\mathscr X)\,,\quad\p\!FH\!\p_{\mathscr B^{r,r}}\,\le\,\p\!F\!\p_{\mathscr B^{p,p}}\p\!H\!\p_{\mathscr B^{q,q}}.
\end{equation*}
Note that for every $r\in[1,\infty]$ there exists $p,q\in[2,\infty]$ such that $1/p+1/q=1/r$ (for example $p=q=2r$)\,. Let us fix such a pair, assuming that $1/s+1/r=1$\,. By Proposition \ref{matraus} we have $\W_{u,w}\in\mathscr B^{p,p}(\mathscr X)$ and $\W_{v,w}^{\,*}\in\mathscr B^{q,q}(\mathscr X)$\,, so by the non-commutative H\"older inequality one gets $\W_{u,w}^{\,*}\,\W_{v,w}\in\mathscr B^{\,r,r}(\mathscr X)$\,. Then by \eqref{monstru}, Remark \ref{trecut} and Proposition \ref{matraus}
$$
\begin{aligned}
\p\!{\sf Ber}_\o(F)\!\p_{\mathbb B[L^2(\G)]}&=\sup_{\p u\p=1=\p v\p}\big\vert\big\<{\sf Ber}_\o(F)u,v\big\>\big\vert\\
&=\sup_{\p u\p=1=\p v\p}\big\vert\big\<F,\W_{u,w}^{\,*}\,\W_{v,w}\big\>_{(\mathscr X)}\big\vert\\
&\le\;\p\!F\!\p_{\mathscr B^{s,s}}\!\sup_{\p u\p=1=\p v\p}\!\p\!\W_{u,w}^{\,*}\,\W_{v,w}\!\p_{\mathscr B^{r,r}}\\
&\le\;\p\!F\!\p_{\mathscr B^{s,s}}\sup_{\p u\p=1}\!\p\!\W_{u,w}^{\,*}\!\p_{\mathscr B^{p,p}}\!\sup_{\p v\p=1}\!\p\!\W_{v,w}\!\p_{\mathscr B^{q,q}}\\
&\le\;\p\!F\!\p_{\mathscr B^{s,s}}\,\p\!\o\!\p^2.
\end{aligned}
$$
Thus a version of \eqref{aceea} has been obtained, but (essentially) weaker if $s\ne\infty$\,.

\medskip
2. Recall that the trace of the product of two positive operators is a positive number: if for instance $A=A_1^* A_1,B=B_1 B^*_1\in\mathbb B(\H_\xi)$\,, then 
$$
{\rm Tr}_\xi(AB)={\rm Tr}_\xi\big(A_1^* A_1 B_1 B_1^*\big)={\rm Tr}_\xi\big[(A_1 B_1)^* (A_1 B_1)\big]\ge 0\,.
$$
Then the statement about the positivity of the Berezin quantization follows from the formula \eqref{viespe} with $u=v$\,.

\medskip
3. To (partly) improve the estimations obtained at 1, we deal now with the trace class properties of the Berezin operator. If $\{w_k\}_{k\in\N}$ is an orthonormal basis in $L^2(\G)$\,, one has by \eqref{viespe}, \eqref{reliable}, the Parseval identity \eqref{hermes} and \eqref{stim}
$$
\begin{aligned}
{\rm Tr}\big[{\sf Ber}_\o(F)\big]&=\sum_k\big\<{\sf Ber}_\o(F)w_k,w_k\big\>\\
&=\sum_k\int_{\mathscr X}\!{\rm Tr}_\xi\!\Big[F(X)\,\big\<{\sf Pr}_\o(X)w_k\!\mid\!w_k\otimes 1_\xi\big\>_\xi\Big]d\mu(X)\\
&=\sum_k\int_{\mathscr X}\!{\rm Tr}_\xi\!\Big[F(X)\,\big\<\O(X)\!\mid\!w_k\otimes 1_\xi\big\>_\xi\,\big\<w_k\otimes 1_\xi\!\mid\!\O(X)\big\>_\xi\Big]d\mu(X)\\
&=\int_{\mathscr X}\!{\rm Tr}_\xi\!\Big[F(X)\,\sum_k\big\<\O(X)\!\mid\!w_k\otimes 1_\xi\big\>_\xi\,\big\<w_k\otimes 1_\xi\!\mid\!\O(X)\big\>_\xi\Big]d\mu(X)\\
&=\int_{\mathscr X}\!{\rm Tr}_\xi\!\Big[F(X)\,\big\<\O(X)\!\mid\!\O(X)\big\>_\xi\Big]d\mu(X)\\
&=\;\p\!\o\!\p^2\!\int_{\mathscr X}\!{\rm Tr}_\xi\!\big[F(X)\big]d\mu(X)\,.
\end{aligned}
$$
If $F$ is positive, we already know that ${\sf Ber}_\o(F)$ is also positive and its trace norm is computed above:
$$
\p\!{\sf Ber}_\o(F)\!\p_{\mathbb B^1[L^2(\G)]}\,={\rm Tr}\big[{\sf Ber}_\o(F)\big]=\,\p\!\o\!\p^2\mathbf{Tr}(F)=\,\p\!F\!\p_{\mathscr B^{1,1}}\,\p\!\o\!\p^2.
$$
One obtains the $s=1$ case of \eqref{aceea} for general $F$ by using the four terms pointwise decomposition valid in each $C^*$-algebra $\mathbb B(\H_\xi)$
$F(X)=\re[F(X)]_+-\re[F(X)]_-+i\im[F(X)]_+-i\im[F(X)]_-\,.$

\medskip
4. The general case in \eqref{aceea} then follows by interpolation from the cases $s=1$ and $s=\infty$\,, because, denoting by $\big[\mathscr Y,\mathscr Z\big]_\theta$ the space of complex interpolation of order $\th\in[0,1]$ of the interpolation pair $(\mathscr Y,\mathscr Z)$\,, one has \cite[Sect. 3.4]{MR}
$\Big[\mathscr B^{\infty,\infty}(\mathscr X),\mathscr B^{1,1}(\mathscr X)\Big]_{1/s}\!=\mathscr B^{s,s}(\mathscr X)$ and $\Big[\mathbb B\big[L^2(\G)],\mathbb B^{1}\big[L^2(\G)\big]\Big]_{1/s}\!=\mathbb B^{s}\big[L^2(\G)\big]$\,.
\end{proof}

\section{Connection with pseudo-differential operators}\label{tronol}

In \cite{MR} many pseudo-differential quantizations $a\mapsto{\sf Op}_{\rm L}^\tau(a)$ and $a\mapsto{\sf Op}_{\rm R}^\tau(a)$ have been introduced, labelled by measurable maps $\tau:\G\to\G$ which are linked to ordering issues. The indices L (left) and R (right) are extra choices connected to the fact that $\G$ is allowed to be non-commutative. We indicated how to pass from one quantization to another. This is also possible for the Berezin formalism, but we decided to treat here only one case. We are going to use here the quantization given formally by
\begin{equation*}\label{op}
[{\sf Op}(a)u](x):=\int_\G\int_{\wG}{\rm Tr}_\xi\big[\xi(x^{-1}y)a(y,\xi)\big]u(y)d\m(y)d\wm(\xi)\,.
\end{equation*}
Among others, it has been shown in \cite{MR} how to define ${\sf Op}$ rigorously as a unitary transformation 
\begin{equation*}\label{transformator}
{\sf Op}:\mathscr B^{2,2}(\G\times\wG):=L^2(\G)\otimes\mathscr B^2(\wG)\to\mathbb B^2\big[L^2(\G)\big]\,.
\end{equation*}
To achieve this one can use our Fourier-Wigner transformation, but it more straightford to define ${\sf Op}$ weakly by
\begin{equation*}\label{dead}
\<{\sf Op}(a)u,v\>:=\big\<a,\mathfrak V_{u,v}\big\>_{\mathscr B^{2,2}(\G\times\wG)}
\end{equation*}
in terms of the obvious scalar product on $\mathscr B^{2,2}(\G\times\wG)$ and {\it the Wigner transform} (connected to $\W_{u,v}$ essentially by a full Fourier transformation)
\begin{equation*}\label{can}
\mathfrak V_{u,v}:=[({\rm id}\otimes\mathscr F)\circ\gamma](\overline u\otimes v)\,,
\end{equation*}
where $[\gamma(g)](x,y)\equiv(g\circ\gamma)(x,y):=g(x,xy^{-1})\,,$ has alredy been used in section \ref{fourtein}.
More explicitly, one has
\begin{equation}\label{grahatt}
\mathfrak V_{u,v}(x,\xi)=\overline{u(x)}\int_\G v(xy^{-1})\xi(y)^*d\m(y)\,.
\end{equation}
Both the partial Fourier transformation and the change of variables operation $\gamma$ are unitary, so one can write
\begin{equation}\label{vocabular}
\<{\sf Op}(a)u,v\>=\big<[({\rm id}\otimes\mathscr F)\circ\gamma]^{-1}a,\overline u\otimes v\big>_{L^2(\G\times\G)}\,.
\end{equation}

On the other hand, let us compute for $F\in\mathscr B^{1,1}(\mathscr X)\cap\mathscr B^{2,2}(\mathscr X)$ and $u,v\in L^2(\G)$
$$
\begin{aligned}
\big\<{\sf Ber}_\o(F)u,v\big\>&=\int_{\mathscr X}\!{\rm Tr}_\xi\!\big[F(X)\,\mathfrak W_{v,\o}(X)^*\,\W_{u,\o}(X)\big]d\mu(X)\\
&=\int_{\mathscr X}\!{\rm Tr}_\xi\!\Big[F(X)\int_\G \overline{v(z)}\,\o(zx^{-1})\xi(z)d\m(z)\int_\G u(y)\,\overline{\o(yx^{-1})}\,\xi(y)^*d\m(y)\Big]d\mu(X)\\
&=\int_\G \int_\G u(y)\overline{v(z)}\,\Big\{\int_{\mathscr X}\!\o(zx^{-1})\,\overline{\o(yx^{-1})}\,{\rm Tr}_\xi\!\big[F(X)\xi(yz^{-1})^*\big]d\mu(X)\Big\}d\m(y)d\m(z)\,.
\end{aligned}
$$
In terms of the scalar product of $L^{2}(\G\times\G)$ this can be written
\begin{equation}\label{writan}
\big\<{\sf Ber}_\o(F)u,v\big\>=\big<K^F_\o,\overline u\otimes v\big>_{L^{2}(\G\times\G)}\,,
\end{equation}
for the kernel
\begin{equation}\label{formulla}
\begin{aligned}
K^F_\o(z,y):=&\int_{\G}\int_{\wG}\o(zx^{-1})\,\overline{\o(yx^{-1})}\,{\rm Tr}_\eta\big[F(\eta,x)\eta(yz^{-1})^*\big]d\m(x)d\wm(\eta)\,.
\end{aligned}
\end{equation}

Therefore, if one wants to have ${\sf Ber}_\o(F)={\sf Op}\big(a^F_\o\big)$\,, taking \eqref{vocabular} and \eqref{writan} into account, the relation
\begin{equation}\label{versionn}
a^F_\o=[({\rm id}\otimes\mathscr F)\circ\gamma]K^F_\o
\end{equation}
must hold. Replacing \eqref{formulla} in \eqref{versionn} one gets
\begin{equation*}\label{flabel}
\begin{aligned}
a^F_\o(q,\xi)
\!=\int_\G \Big\{\int_{\G}\int_{\wG}\o(qx^{-1})\,\overline{\o(qs^{-1}x^{-1})}\,{\rm Tr}_\eta\big[F(\eta,x)\eta(qs^{-1}q^{-1})^*\big]d\m(x)d\wm(\eta)\Big\}\,\xi(s)^*d\m(s)\,.
\end{aligned}
\end{equation*}
When $\G$ is Abelian, using additive notations in $\G$\,, multiplicative notations in the Pontryagin dual group $\wG$ and setting $F_\bullet(q,\xi):=F(\xi,q)$\,, this can be written as a convolution in $\G\times\wG$\,:
\begin{equation*}\label{ashwulah}
\begin{aligned}
a^F_\o(q,\xi)&=\!\int_{\G}\int_{\wG}F(\eta,x)\Big[\o(q-x)\!\int_\G\,\overline{\o(q-x-s)}\,\overline{(\eta^{-1}\xi)(s)}d\m(s)\Big]d\m(x)\d\wm(\eta)=(F_\bullet\ast\mathfrak V_{\overline\o,\overline\o})(q,\xi)\,.
\end{aligned}
\end{equation*}

\section{Toeplitz-like operators in the Bargmann representation}\label{mirlitronol}

\begin{Definition}\label{coybirent}
Let us fix an element $\o\in L^2(\G)$\,. 
Taking Proposition \ref{mataus} into account, we define $\,\W_{\o}:L^2(\G)\rightarrow \mathscr B^{2,2}(\mathscr X)$ by
\begin{equation*}\label{sambur}
\big[\W_{\o}(u)](X):=\W_{u,\o}(X)=\<\We(X)^*u\!\mid\!\o\otimes 1_\xi\>_\xi
\end{equation*}
and call it {\rm the generalized Bargmann transformation associated to the vector $\o$\,}.
\end{Definition}

Using \eqref{wilhelm} and \eqref{stass}, it is possible to write
\begin{equation*}\label{gristos}
\big[\W_{\o}(u)](X)=\int_\G u(z)\,\overline{\o(zx^{-1})\,}\xi(z)^*d\m(z)=\int_\G u(z)\,[\O(X)](z)^*d\m(z)\,.
\end{equation*}

\begin{Proposition}\label{ishen}
\begin{enumerate}
\item[(i)]
The adjoint $\W_{\o}^\dagger:\mathscr B^{2,2}(\mathscr X)\rightarrow L^2(\G)$ is given on $\mathscr B^{1,1}(\mathscr X)\cap\mathscr B^{2,2}(\mathscr X)$ by 
\begin{equation*}\label{bivol}
\big[\W_{\o}^\dagger(F)\big](z)=\int_\mathscr X\!{\rm Tr}_\xi\Big\{F(X)\,[\O(X)](z)\Big\}\,d\mu(X)\,.
\end{equation*}
\item[(ii)]
The final projection $\mathfrak P_\o=\W_{\o}\W_{\o}^\dagger$ of the isometry $\,\W_{\o}:L^2(\G)\rightarrow \mathscr B^2(\mathscr X)$ is given on $\mathscr B^{1,1}(\mathscr X)\cap\mathscr B^{2,2}(\mathscr X)$ by
\begin{equation*}\label{proy}
\big[\mathfrak P_\o(F)\big](X)=\int_\mathscr X{\rm TR}_\eta\Big\{p_\o(Y,X)\big[F(Y)\otimes 1_\xi\big]\Big\}\,d\mu(Y)\,,
\end{equation*}
where $X=(\xi,x)\,,\,Y=(\eta,y)$\,, 
$$
{\rm TR}_\eta:={\rm Tr}_\eta\otimes{\sf id}_\xi:\mathbb B(\H_\eta)\otimes\mathbb B(\H_\xi)\to\mathbb C\otimes\mathbb B(\H_\xi)\equiv\mathbb B(\H_\xi)
$$ 
is a "partial trace" and
\begin{equation*}\label{nicleu}
p_\o(Y,X):=\int_\G\,[\O(Y)](z)\otimes[\O(X)](z)^*d\m(z)\in\mathbb B(\H_\eta)\otimes\mathbb B(\H_\xi)\,.
\end{equation*}
\end{enumerate}
\end{Proposition}

\begin{proof}
(i) For $F\in\mathscr B^{1,1}(\mathscr X)\cap\mathscr B^{2,2}(\mathscr X)$ we have
$$
\begin{aligned}
\big\<\W_{\o}(u),F\big\>_{\mathscr B^{2,2}}&=\int_\G\int_{\wG}\,{\rm Tr}_\xi\big\{\W_{u,\o}(\xi,x)F(\xi,x)^*\big\}d\m(x)d\wm(\xi)\\
&=\int_\G\int_{\wG}\,{\rm Tr}_\xi\Big\{\int_\G u(z)\,[\O(\xi,x)](z)^*d\m(z)\,F(\xi,x)^*\Big\}d\m(x)d\wm(\xi)\\
&=\int_\G u(z)\,\Big[\int_\G\int_{\wG}{\rm Tr}_\xi\big\{\,[\O(\xi,x)](z)^*\,F(\xi,x)^*\big\}d\m(x)d\wm(\xi)\Big]\,d\m(z)\\
&=\int_\G u(z)\,\Big[\,\overline{\int_\G\int_{\wG}{\rm Tr}_\xi\big\{\,F(\xi,x)\,[\O(\xi,x)](z)\big\}d\m(x)d\wm(\xi)}\Big]\,d\m(z)\,.
\end{aligned}
$$

(ii) One has
$$
\begin{aligned}
\big[\W_{\o}\big(\W_{\o}^\dagger F\big)\big](X)&=\W_{\W_{\o}^\dagger\Psi,\o}(X)=\int_\G \big(\W_{\o}^\dagger F\big)(z)\,[\O(X)](z)^*d\m(z)\\
&=\int_\G \Big[\int_\mathscr X\!{\rm Tr}_\eta\Big\{F(Y)[\O(Y)](z)\Big\}\,d\mu(Y)\Big]\,[\O(X)](z)^*d\m(z)\\
&= \int_\mathscr X\int_\G\,{\rm Tr}_\eta\Big\{F(Y)[\O(Y)](z)\Big\}[\O(X)](z)^*d\m(z)\,d\mu(Y)\\
&= \int_\mathscr X{\rm TR}_\eta\Big\{\big[F(Y)\otimes 1_\xi\big]\!\int_\G\,[\O(Y)](z)\otimes[\O(X)](z)^*d\m(z)\Big\}\,d\mu(Y)\,.
\end{aligned}
$$
\end{proof}

Since $\,\W_{\o}^\dagger\W_{\o}=1$\,, one has {\it the inversion formula} (valid almost everywhere)
\begin{equation*}\label{bufal}
u(q)=\int_\mathscr X\!{\rm Tr}_\xi\Big\{[\W_{\o}(u)](X)\,[\O(X)](q)\Big\}\,d\mu(X)\,.
\end{equation*}
and {\it the reproducing formula} $\W_{\o}(u)=\mathfrak P_\o\!\[\W_{\o}(u)\]$\,, i.e.
\begin{equation*}\label{taur}
\[\W_{\o}(u)\]\!(X)=\int_\mathscr X{\rm TR}_\eta\Big\{p_\o(Y,X)\big([\W_{\o}(u)](Y)\otimes 1_\xi\big)\Big\}\,d\mu(Y)\,.
\end{equation*}

\begin{Remark}\label{sturion} 
{\rm If $\G$ is Abelian $\mathscr P_\o(\mathscr X):=\mathfrak P_\o\!\[L^2(\wG\times\G)\]$ is a reproducing space with reproducing kernel $(X,Y)\mapsto p_\o(X,Y)=\<\O(Y),\O(X)\>$\,; it is composed of bounded continuous functions on $\wG\times\G$\,.}
\end{Remark}

We give now a Toeplitz-like form of the operator ${\sf T\!p}_\o(F):=\W_\o\circ{\sf Ber}_\o(F)\circ\(\W_\o\)^{\dag}$ living in $\mathscr B^{2,2}(\mathscr X)$\,.

\begin{Proposition}\label{carabus}
One has
\begin{equation}\label{ragace}
{\sf T\!p}_{\o}(F)=\mathfrak P_\o\circ{\sf Diag}_R(F)\circ \mathfrak P_\o\,,
\end{equation}
where ${\sf Diag}_R(F)$ is the diagonalizable operator in the Hilbert space $\mathscr B^{2,2}(\mathscr X)$\,, defined by pointwise right multiplication by $F\in\mathscr B^{\infty,\infty}(\mathscr X)$\,.
\end{Proposition}

\begin{proof}
Clearly \eqref{ragace} is equivalent to ${\sf Ber}_\o(F)=\(\W_\o\)^{\dag}\!\circ{\sf Diag}_R(F)\circ\W_\o$\,. For $u,v\in L^2(\G)$ we have
$$
\begin{aligned}
\big\<{\sf Ber}_{\o}(F)u,v\big\>&=\big\<F,\W_\o(u)^*\,\W_\o(v)\big\>_{(\mathscr X)}\\
&=\int_\mathscr X\!{\rm Tr}_\xi\Big\{\big[\W_\o(u)\big](X)F(X)\big[\W_\o(v)\big](X)^*\Big\}\,d\mu(X)\\
&=\int_\mathscr X\!{\rm Tr}_\xi\Big\{\Big({\sf Diag}_R(F)\big[\W_\o(u)\big]\Big)(X)\,\big[\W_\o(v)\big](X)^*\Big\}\,d\mu(X)\\
&=\Big\<{\sf Diag}_R(F)\big[\W_\o(u)\big],\W_\o(v)\Big\>_{(\mathscr X)}\\
&=\Big\<\big[\(\W_\o\)^{\dag}\!\circ{\sf Diag}_R(F)\circ\W_\o\big]u,v\Big\>
\end{aligned}
$$
and the Proposition is proved.
\end{proof}

\vspace{\baselineskip}

\bf{Acknowledgement:} The author was supported by N\'ucleo Milenio de F\'isica Matem\'atica RC120002 and the Fondecyt Project 1120300.

\vspace{\baselineskip}

\end{document}